\documentclass[12pt,a4paper]{article}
\usepackage[utf8]{inputenc}

\author{Ivo Slegers}
\title{Equivariant harmonic maps depend real analytically on the representation}

\usepackage{amsmath}
\usepackage{amsfonts}
\usepackage{amsthm}
\usepackage{mathtools}
\usepackage{enumitem}
\usepackage{hyperref}
\usepackage{cleveref}

\newcommand{\N}{\mathbb{N}}

\newcommand{\R}{\mathbb{R}}
\newcommand{\C}{\mathbb{C}}

\DeclarePairedDelimiter{\norm}{\lVert}{\rVert}

\DeclarePairedDelimiter{\abs}{\lvert}{\rvert}

\DeclarePairedDelimiterX{\inner}[2]{\langle}{\rangle}{#1, #2}

\DeclareMathOperator{\Rep}{Rep}
\DeclareMathOperator{\Harm}{Harm}
\DeclareMathOperator{\Hom}{Hom}
\DeclareMathOperator{\GL}{GL}
\DeclareMathOperator{\SL}{SL}
\DeclareMathOperator{\SO}{SO}
\DeclareMathOperator{\tr}{trace}
\DeclareMathOperator{\vol}{vol}
\DeclareMathOperator{\id}{id}
\DeclareMathOperator{\pr}{pr}
\DeclareMathOperator{\im}{im}

\DeclareMathOperator{\HomHit}{Hom_{Hit}}

\newcommand{\ddt}[1][]{\frac{\partial}{\partial t_{#1}}}
\newcommand{\parder}[2]{\frac{\partial #1}{\partial #2}}

\renewcommand{\H}{\mathbb{H}}

\newtheorem{theorem}{Theorem}[section]
\newtheorem{proposition}[theorem]{Proposition}
\newtheorem{corollary}[theorem]{Corollary}
\newtheorem{lemma}[theorem]{Lemma}
\newtheorem{definition}[theorem]{Definition}

\newtheoremstyle{remarkstyle}
{}
{}
{}
{}
{\bfseries}
{.} 
{7pt}
{}

\theoremstyle{remarkstyle}
\newtheorem{remark}[theorem]{Remark}

\begin{document}
\maketitle

\begin{abstract}
We prove that when assuming suitable non-degeneracy conditions equivariant harmonic maps into symmetric spaces of non-compact type depend in a real analytic fashion on the representation they are associated to. The main tool in the proof is the construction of a family of deformation maps which are used to transform equivariant harmonic maps into maps mapping into a fixed target space so that a real analytic version of the results in \cite{EellsLemaire} can be applied.
\end{abstract}

\section{Introduction}\label{sec:introduction}

In this article we prove that equivariant harmonic maps into symmetric spaces of non-compact type  depend in a real analytic fashion on the representation they are associated to. Throughout this article we let $M$ be a closed real analytic Riemannian manifold, $\widetilde{M}$ its universal cover and $\Gamma = \pi_1(M)$ its fundamental group. Also we let $G$ be a real semisimple Lie group without compact factors and $X=G/K$ its associated symmetric space. If $\rho \colon \Gamma \to G$ is a reductive representation of $\Gamma$ in $G$ then by work of Corlette (\cite{Corlette}) there exists a $\rho$-equivariant harmonic map $f \colon \widetilde{M} \to X$. A map $f$ is called $\rho$-equivariant if
\[
f(\gamma m) = \rho(\gamma) f(m) \text{ for all } m \in \widetilde{M} \text{ and } \gamma\in \Gamma.
\]
These maps were used by Corlette to prove a version of super rigidity. They were also used by Hitchin and Simpson in the development of the Non-Abelian Hodge correspondence which gives an identification between representation varieties and moduli spaces of Higgs bundles over K\"ahler manifolds. 

In \cite{EellsLemaire} Eells and Lemaire proved that, under suitable non-degeneracy conditions, harmonic maps between closed Riemannian manifolds depend smoothly on the metrics on both the domain and the target (see also \cite{Sampson}). Similarly one expects that equivariant harmonic maps depend smoothly (or even real analytically) on the representation when a similar non-degeneracy condition is imposed. The purpose of the current article is to prove that this is indeed true.

The main result (\Cref{thm:mainresult}) of this article is as follows. If $(\rho_t)_t$ is a real analytic family of representations of $\Gamma$ in $G$ such that $\rho_0$ is reductive and the centraliser of its image contains no semi-simple elements then for all $t$ in a neighbourhood of zero there exist $\rho_t$-equivariant harmonic maps depending real analytically on $t$. Similarly we also prove real analytic dependence on the metric on the domain $M$. Furthermore we prove a real analytic version of the results in \cite{EellsLemaire} which will serve as the central analytic ingredient in the proof of the main theorem (see \Cref{prop:implicitfunctiontheorem}).

In \Cref{subsec:hitchinreps} we apply these results to families of Hitchin representations. Such families satisfy the assumptions of the main theorem (see \Cref{prop:realanalyticityhitcin}). As a result we can characterise certain sets as real analytic subsets of Teichm\"uller space and the set of Hitchin representations. Namely in \Cref{cor:nonimmersionisrealanalyticvar} we prove that the set of points at which the equivariant harmonic maps are not immersions is a real analytic subvariety of the universal Teichm\"uller curve crossed with the set of Hitchin representations. Similarly in \Cref{cor:minimalsurfsisrealanalyticvar} we prove that the set of points $Y$ in Teichm\"uller space and representations $\rho$ such that $Y$ can be realised as a minimal surface in $X/\rho(\Gamma)$ is a real analytic subvariety of Teichm\"uller space crossed with the set of Hitchin representations.\\\\
\textbf{Acknowledgements.} The author would like to thank Prof. Ursula Hamenst\"adt for suggesting this problem and her help and encouragement during the project. The author is also grateful for the financial support received from the Max Planck Institute for Mathematics.
\section{Statement of the results}\label{sec:statementofresults}
We first collect some preliminary definitions and results needed to give a statement of the main theorem.
\subsection{Harmonic maps}
If $(M,g)$ and $(N,h)$ are Riemannian manifolds with $M$ compact then a $C^1$ map $f \colon (M,g) \to (N,h)$ is called harmonic if it is a critical point of the Dirichlet energy functional
\[
E(f) = \frac{1}{2} \int_{M} \norm{df}^2 \vol_g.
\]
A harmonic map satisfies the Euler--Lagrange equation $\tau(f) = 0$ where $\tau(f) = \tr_{g} \nabla df$ is the tension field of $f$. If the domain $M$ is not compact a map is called harmonic if its tension field vanishes identically.

At a critical point the Hessian of the energy functional is given by
\[
\nabla^2 E(f)(X,Y) = \int_{M} [\inner{\nabla X}{\nabla Y} - \tr_g \inner{R^N(X, df \cdot) df \cdot}{Y}] \vol_g
\]
for $X,Y \in \Gamma^1(f^*TN)$. Here $\nabla$ denotes the connection on $f^*TN$ induced by the Levi--Civita connection on $TN$ and $R^N$ denotes the Riemannian curvature tensor of $(N,h)$. The non-degeneracy condition imposed on harmonic maps in \cite{EellsLemaire} is that a harmonic map $f$ is a non-degenerate critical point of the energy i.e. $\nabla^2 E(f)$ is a non-degenerate bilinear form. In \cite{Sunada} Sunada proved that if the target is a locally symmetric space of non-positive curvature this condition is satisfied if and only if the harmonic map is unique. We collect these non-degeneracy conditions in the following lemma.
\begin{lemma}[Sunada]\label{lem:nondegeneracy} Suppose $N = X/\Lambda$ (with $\Lambda \subset G$) is a locally symmetric space of non-positive curvature and $f \colon M \to N$ a harmonic map then the following are equivalent:
\begin{enumerate}[label = \roman*.]
\item $f$ is a non-degenerate critical point of $E$.
\item $f$ is the unique harmonic map in its homotopy class.
\item The centraliser of $\Lambda$ in $G$ contains no semi-simple elements.
\end{enumerate}
\end{lemma}
\begin{proof}
We can write $\nabla^2E(f) =- \int_M \inner{J_f \cdot}{\cdot} \vol_g$ where 
\[
J_f(X) = \tr_g \nabla^2 X + \tr_g R^N(X, df\cdot)df\cdot
\]
is the Jacobi operator at $f$. As discussed in \cite[p.35]{EellsLemaire} the Hessian $\nabla^2 E(f)$ is non-degenerate precisely when $\ker J_f = 0$. It follows from \cite[Proposition 3.2]{Sunada} that the set 
\[
\Harm(M,N, f) := \{h \colon M \to N \mid h \text{ is harmonic and homotopic to } f\}
\]
is a submanifold of $W^k(M,N)$ (the space of maps from $M$ to $N$ equipped with the $W^{k,2}$ Sobolev topology) with tangent space at $h$ equal to $\ker J_h$. Because $N$ has non-positive curvature the space $\Harm(M,N,f)$ is connected (\cite[Theorem 8.7.2]{Jost}). We see that $f$ is a non-degenerate critical point of the energy if and only if $\Harm(M,N,f)$ contains only $f$. If $g \in G$ is a semi-simple isometry centralising $\Lambda$ then it is clear that $h=g\cdot f$ is a distinct harmonic map homotopic to $f$. Conversely, if $h$ is a harmonic map homotopic to $f$ then there exists a semi-simple $g\in G$ contained in the centraliser of $\Lambda$ such that $h = g\cdot f$ by \cite[Lemma 3.4]{Sunada}. We conclude that $f$ is the unique harmonic map in its homotopy class if and only if the centraliser of $\Lambda$ in $G$ contains no semi-simple elements.
\end{proof}
The main existence result in the theory of equivariant harmonic maps is the following theorem of Corlette.
\begin{proposition}[\cite{Corlette}]\label{prop:existenceharmonicmaps}
A representation $\rho \colon \Gamma \to G$ is reductive if and only if there exists a $\rho$-equivariant harmonic map $f\colon \widetilde{M} \to X$.
\end{proposition}
A representation $\rho \colon \Gamma \to G$ is called reductive if the Zariski closure of it image in $G$ is a reductive subgroup.
\subsection{Families of representations and metrics}
We will index families of representations or metrics by open balls in $\R^n$. For $\epsilon>0$ we denote by $D_\epsilon$ the open $\epsilon$ ball in $\R^n$ centred at $0$. Let $(\rho_t)_{t\in D_\epsilon}$ be a family of representations $\rho_t \colon \Gamma \to G$. Such a family induces a natural action of $\Gamma$ on $X\times D_\epsilon$ given by $\gamma \cdot (x,t) = (\rho_t(\gamma)x, t)$. We make the following properness and freeness assumption on the families of representations we will consider.
\begin{definition}\label{def:uniformlyfreeandproper}
We call a family of representations uniformly free and proper if the induced action on $X\times D_\epsilon$ is free and proper.
\end{definition}
In particular, each representation in such a family acts freely and properly on $X$. On families of representations we will make the following regularity assumption.
\begin{definition}\label{def:analyticfamilyrepresentations}
A family of representations $(\rho_t)_{t\in D_\epsilon}$ of $\Gamma$ in $G$ is called real analytic if for every $\gamma\in \Gamma$ the map $D_\epsilon \to G \colon t \mapsto \rho_t(\gamma)$ is real analytic.
\end{definition}
\begin{remark}
A family of representations can be seen as a map from $D_\epsilon$ into $\Hom(\Gamma, G)$, the set of representations of $\Gamma$ into $G$. If $G$ is an algebraic subgroup of $\GL(n,\R)$ and if $S$ is a generating set of $\Gamma$ with relations $R$ then $\Hom(\Gamma,G)$ can be realised as the closed subset of $\GL(n,\R)^S$ consisting of tuples $(g_1,...,g_n)$ satisfying the relations $r(g_1,...,g_n)=1$ for $r\in R$. In this way we realise $\Hom(\Gamma,G)$ as a real algebraic variety. We note that in this case a family of representations is real analytic if and only if the map $D_\epsilon \to \Hom(\Gamma, G)$ is real analytic.
\end{remark}

Finally, on families of metrics we make the following regularity assumptions.
\begin{definition}\label{def:analyticfamilymetrics}
We call a family $(g_t)_{t\in D_\epsilon}$ of Riemannian metrics on $M$ a real analytic family of metrics if $(x,t) \mapsto g_t(x)$ induces a real analytic map $M\times D_\epsilon \to S^2 T^*M$.
\end{definition}

\subsection{Mapping spaces}
If $M$ and $N$ are real analytic manifolds we denote by $C^{k,\alpha}(M,N)$ ($k\in \N, 0<\alpha<1$) the space of $k$-times differentiable maps from $M$ to $N$ such that the $k$th derivatives are $\alpha$-H\"older continuous. We equip these spaces with the topology of uniform $C^{k,\alpha}$ convergence on compact sets. If the domain manifold $M$ is compact then these spaces can be equipped with a natural real analytic Banach manifold structure.

There is no such Banach manifold structure when $M$ is not compact. It is possible to instead give a Fr\'echet manifold structure where a chart around a point $f \colon M \to N$ is modelled on spaces of sections of $f^*TN$ with compact support. Such structures are not useful when considering convergence of equivariant maps $\widetilde{M} \to X$ because variations will necessarily not be compactly supported. We will instead make use of the fact that equivariant maps are determined by their values on a fundamental domain which allows us to state our results using Banach manifolds after all.

When $M$ is a closed manifold we let $\Omega \subset \widetilde{M}$ be a bounded domain containing a fundamental domain for the action of $\Gamma$ on $\widetilde{M}$. We note that a map $\widetilde{M} \to X$ that is equivariant with respect to any representation is completely determined by its restriction to $\Omega$. Furthermore, $\rho_n$-equivariant maps $f_n$ converge to a $\rho$-equivariant map $f$ uniformly on compacts if and only if the restrictions $f_n\vert_\Omega$ converge uniformly to $f\vert_\Omega$.

We will consider the space of bounded functions from $\Omega$ to $X$. For this we equip $M$ with a background metric and for simplicity we identify $X$ with $\R^n$ via the exponential map $\exp_o \colon T_o X \to X$ based at some basepoint $o\in X$. The metric on $M$ induces a $C^{k,\alpha}$ norm on the space of functions $\Omega \to \R^n \cong X$. We denote by $C^{k,\alpha}_b(\Omega, X)$ the space of functions for which this norm is bounded. This space can be equipped with the structure of a real analytic Banach manifold. For this we observe that equipped with the $C^{k,\alpha}$ norm the space $C^{k,\alpha}_b(\Omega, X)$ becomes a Banach space (note that the linear structure comes from the identification $X \cong \R^n$ and carries no direct geometric meaning). The Banach manifold structure is obtained by declaring this to be a global chart. One can check that the Banach manifold structure is independent of the choice of background metric on $M$ and basepoint in $X$ (see \Cref{lem:omegalemma}). We would like to note that, although the use of the identification $X \cong \R^n$ is somewhat ad hoc, if we replace the domain $\Omega$ by a closed manifold $M$ then the above construction yields the usual Banach manifold structure on the space $C^{k,\alpha}(M,X)$.
\subsection{Main result}
The main result of this article can be stated as follows.
\begin{theorem}\label{thm:mainresult}
Let $(g_t)_{t\in D_\epsilon}$ be a real analytic family of metrics on $M$ and let $(\rho_t)_{t\in D_\epsilon}$ be a real analytic family of representations of $\Gamma$ in $G$. We assume that the family $(\rho_t)_{t\in D_\epsilon}$ is uniformly free and proper. Suppose $\rho_0$ is reductive and the centraliser $Z_G(\im \rho_0)$ contains no semi-simple elements. Then for every $k \in \N, 0< \alpha<1$ there exists a $\delta>0$ smaller then $\epsilon$ and a unique continuous map $F \colon D_\delta \to C^{k,\alpha}(\widetilde{M}, X)$ such that each $F(t)$ is a $\rho_t$-equivariant harmonic map $(\widetilde{M}, g_t)\to X$ and the restricted map $F(\cdot)\vert_{\Omega} \colon D_\delta \to C^{k,\alpha}_b(\Omega, X)$ is real analytic.
\end{theorem}
\begin{remark}
The above result is also true in the smooth category. More precisely we can define, analogous to Definitions \ref{def:analyticfamilyrepresentations} and \ref{def:analyticfamilymetrics}, the notion of smooth families of metrics and representations. Then \Cref{thm:mainresult} also holds when we replace `real-analytic' by `smooth'. For brevity we will not prove the smooth case here but the reader can easily check that the proof goes through also in this case.
\end{remark}

If each $\rho_t$ is reductive and has trivial centraliser then by applying the above theorem at each $t\in D_\epsilon$ we obtain immediately the following corollary.
\begin{corollary}\label{cor:nondegeneratefamily}
Let $(g_t)_{t\in D_\epsilon}$, $(\rho_t)_{t\in D_\epsilon}$ be as in \Cref{thm:mainresult}. Suppose that for every $t\in D_\epsilon$ the representation $\rho_t$ is reductive and $Z_G(\im \rho_t) = 0$. Then for all $k\in \N, 0 < \alpha < 1$ there exists a unique continuous map $F \colon D_\epsilon \to C^{k,\alpha}(\widetilde{M}, X)$ such that each $F(t) \colon (\widetilde{M},g_t)\to X$ is a $\rho_t$-equivariant harmonic map and the restricted map $F(\cdot)\vert_\Omega \colon D_\epsilon \to C^{k,\alpha}_b(\Omega,X)$ is real analytic.
\end{corollary}
\subsubsection{Hitchin representations}\label{subsec:hitchinreps}
We briefly mention how the above results can be applied when we consider Hitchin representations. In this section we let $M = S$ be a closed surface of genus $g \geq 2$ and as before $\Gamma = \pi_1(S)$. In this case the harmonicity of a map $f\colon S \to N$ depends only on the conformal class of the metric on $S$. Also in this section we let $G$ be a split real Lie group.

In \cite{Hitchin} Hitchin proved that each representation variety $\Rep(\Gamma, G) = \Hom_{\text{red}}(\Gamma,G)/G$ contains a connected component, now called the Hitchin component, which contains $\mathcal{T}(S)$, the Teichm\"uller space of $S$, in a natural way. We denote by $\HomHit(\Gamma,G)$ the component of $\Hom(\Gamma, G)$ consisting of Hitchin representations i.e. representations of $\Gamma$ in $G$ whose conjugacy class lies in the Hitchin component.

Hitchin representations enjoy many special properties. Relevant to our discussion is that they are reductive and the centraliser of their image is trivial. Also, in \cite{Labourie} Labourie showed that Hitchin representations act freely and are Anosov representations. It follows from \cite[Theorem 7.33]{KLP} that continuous families of Anosov representations satisfy the uniformly free and proper assumption as in \Cref{def:uniformlyfreeandproper}.

It follows that there exists a map $F \colon \mathcal{T}(S) \times \HomHit(\Gamma,G) \to C^{k,\alpha}(\widetilde{S},X)$ assigning to each $(J, \rho)$ the unique $\rho$-equivariant harmonic map $(\widetilde{S},J) \to X$. A chart of $\HomHit(\Gamma,G)$ modelled on $D_\epsilon$ can be seen as real analytic family of representations that is uniformly free and proper. Furthermore it follows from \cite{Wolf} that it is possible to choose metrics $g_J$ on $S$ representing points in Teichm\"uller space $J\in \mathcal{T}(S)$ depending on $J$ in a real analytic fashion. By applying \Cref{cor:nondegeneratefamily} to charts around points $(J,\rho) \in \mathcal{T}(S) \times \HomHit(\Gamma,G)$ we obtain the following proposition.
\begin{proposition}\label{prop:realanalyticityhitcin}
For all $k\in \N, 0 < \alpha < 1$ the map 
\[
F \colon \mathcal{T}(S) \times \HomHit(\Gamma,G) \to C^{k,\alpha}(\widetilde{S}, X)
\]
assigning to each $(J,\rho)$ the unique $\rho$-equivariant harmonic map $(\widetilde{S},J) \to X$ is continuous and the restricted map $F(\cdot,\cdot)\vert_\Omega \colon \mathcal{T}(S) \times \HomHit(\Gamma,G) \to C_b^{k,\alpha}(\Omega,X)$ is real analytic.
\end{proposition}
We discuss three corollaries to this result.

First we observe that the family of harmonic maps given by $F$ can also be interpreted as a single map with the universal Teichm\"uller curve as domain. Namely let $\Xi(S)$ be the universal Teichm\"uller curve of $S$. It is a trivial fibre bundle over $\mathcal{T}(S)$ with fibres homeomorphic to $S$. It is equipped with a complex structure such that the fibre $\Xi(S)_J$ over $J \in \mathcal{T}(S)$ together with the marking provided by the trivialization $\Xi(S)_J \cong \mathcal{T}(S)\times S$ determines the point $J$ in Teichm\"uller space (see \cite[section 6.8]{Hubbard}). The universal cover $\widetilde{\Xi}(S)$ is a trivial fibre bundle over $\mathcal{T}(S)$ with fibres homeomorphic to $\widetilde{S}$. Let $F' \colon \widetilde{\Xi}(S)\times \HomHit(\Gamma,G) \to X$ be the map which on each fibre $\widetilde{\Xi}(S)_J \times\{\rho\} \cong \widetilde{S}$ is given by the $\rho$-equivariant harmonic map $(\widetilde{\Xi}(S)_J, J) \to X$. It follows from \Cref{prop:realanalyticityhitcin} that this map is real analytic.  
\begin{corollary}\label{cor:nonimmersionisrealanalyticvar}
The set
\begin{align*}
I &= \{((J,x),\rho) \in \widetilde{\Xi}(S)\times \HomHit(\Gamma,G) \mid\\ &F'(J,\cdot, \rho) \colon \widetilde{\Xi}(S)_J \to X \text{ is not an immersion at } x\}
\end{align*}
is a real analytic subvariety of $\widetilde{\Xi}(S)\times \HomHit(\Gamma,G)$.
\end{corollary}
It is conjectured (see for example \cite[Conjecture 9.3]{Qiongling}) that equivariant harmonic maps associated to Hitchin representations are immersions which would correspond to the set $I$ being empty.
\begin{proof}
We equip $\Xi(S)$ with a choice of real analytic metric. Given a point $((J,x),\rho) \in \widetilde{\Xi}(S) \times \HomHit(\Gamma,G)$ we consider the derivative of $F'$ in the fibre direction
\[
dF(J,\cdot,\rho) \colon T_x(\widetilde{\Xi}(S)_J) \to \text{im} dF(J,\cdot, \rho).
\]
Because the spaces $T_x(\widetilde{\Xi}(S)_J)$ and $\text{im} dF(J,\cdot,\rho) \subset T_{F'(J,x,\rho)}X$ are equipped with inner products we can consider the determinant of this map. We let $d \colon \widetilde{\Xi}(S)\times \HomHit(\Gamma,G)\to\R$ be the map which at a point $((J,x),\rho)$ is given by the determinant of the above map. Because $F'$ is real analytic the map $d$ is real analytic as well. We observe that $I = d^{-1}(0)$ from which the result follows.
\end{proof}
In a similar vein we also have the following corollary.
\begin{corollary}\label{cor:minimalsurfsisrealanalyticvar}
The set
\begin{align*}
T&=\{(J,\rho) \in \mathcal{T}(S)\times \HomHit(\Gamma,G) \mid\\ &(S,J) \text{ can be realised in } X/\rho(\Gamma) \text{ as a branched minimal surface}\}
\end{align*}
is a real analytic subvariety of $\mathcal{T}(S)\times \HomHit(\Gamma,G)$.
\end{corollary}
\begin{proof}
For $J\in \mathcal{T}(S)$ and $\rho\in \HomHit(\Gamma,G)$ we consider the Hopf differential of the harmonic map $F(J,\rho) \colon (\widetilde{S}, J) \to X$, namely $\phi_{J,\rho} = (F(J,\rho)^* m)^{2,0}$. Here $m$ is the Riemannian metric of the symmetric space $X$. The Hopf differential is a holomorphic quadratic differential on $(\widetilde{S}, J)$ which vanishes if and only if the harmonic map $F(J,\rho)$ is a (branched) minimal surface. The Hopf differential $\phi_{J,\rho}$ is $\Gamma$-invariant and descends to $S$ since $F(J,\rho)$ is $\rho$-equivariant. Consider the function $V \colon \mathcal{T}(S)\times \HomHit(\Gamma,G) \to \R$ given by the $L^2$-norm of $\phi_{J,\rho}$, namely
\[
V(J,\rho) = \int_S \frac{\abs{\phi_{J,\rho}}^2}{\sqrt{\det g_J}} \abs{dz}^2.
\]
Here $g_J$ is a metric in the conformal class of $J$ depending real analytically on $J$. It follows from \Cref{prop:realanalyticityhitcin} that this function is real analytic (it is for this reason that we choose the $L^2$-norm rather than the $L^1$-norm). The statement of the corollary follows from $T = V^{-1}(0)$.
\end{proof}

The space of Fuchsian representations $\Gamma \to \SL(2,\R)$ can be included in $\HomHit(\Gamma,G)$ in the following way. There exists a representation $\iota \colon \SL(2,\R) \to G$ that is unique up to conjugation. Then to a Fuchsian representation $\rho_0 \colon \Gamma \to \SL(2,\R)$ we associate a, so called, quasi-Fuchsian representation $\rho = \iota \circ \rho_0$ which lies in $\HomHit(\Gamma,G)$. This inclusion descends to the natural inclusion of Teichm\"uller space into the Hitchin component. A quasi-Fuchsian representation stabilises the totally geodesically embedded copy of $\H^2$ in $G/K$ given by the inclusion $\iota' \colon \SL(2,\R)/\SO(2) \to G/K= X$ that is induced by $\iota$. By uniqueness of equivariant harmonic maps we have that the harmonic map $(\widetilde{S},J) \to X$ equivariant for $\rho = \iota \circ \rho_0$ is given by the composition $\iota'\circ f_0$ where $f_0 \colon (\widetilde{S},J) \to \H^2$ is the unique $\rho_0$-equivariant map. It follows from \cite[Theorem 11]{Sampson} that $f_0$ is a diffeomorphism. Hence equivariant harmonic maps associated to quasi-Fuchsian representations are immersions. Because being an immersion is an open condition we immediately obtain the following corollary to \Cref{prop:realanalyticityhitcin}
\begin{corollary}
There exists an open neighbourhood of 
\[
\mathcal{T}(S)\times \{\text{quasi-Fuchsian representations}\} \subset \mathcal{T}(S)\times \HomHit(\Gamma,G)
\]
such that for any pair $(J,\rho)$ in this neighbourhood the $\rho$-equivariant harmonic map $(\widetilde{S},J) \to X$ is an immersion.
\end{corollary}
\section{Proof of the main result}

As in the proof of Eells and Lemaire in \cite{EellsLemaire}, our main analytical tool will be the implicit function theorem for maps between Banach manifolds. The main difficulty to overcome is that a priori the equivariant harmonic maps are not elements of the same space of mappings. Namely, if $(\rho_t)_t$ is a family of representations then a $\rho_t$-equivariant map is an element of the space $C^{k,\alpha}(M, X/\rho_t(\Gamma))$. Since the target manifold is different for each $t$ these spaces are not equal (although they are likely to be diffeomorphic). Our aim is to modify these maps so that they can be seen as elements of a single mapping space. This will be achieved by means of a family of deformation maps which intertwine the representations $\rho_0$ and $\rho_t$. By composing with these deformation maps we can view each $\rho_t$-equivariant map as element of (a subset of) $C^{k,\alpha}(M, X/\rho_0(\gamma))$.

We first fix some notation. We let $(\rho_t)_{t\in D_\epsilon}$ be a real analytic family of representations that is uniformly free and proper. We denote $X_\epsilon = X\times D_\epsilon$ and by $\alpha \colon \Gamma \times X_\epsilon \to X_\epsilon, \alpha(\gamma)(x,t) = (\rho_t(\gamma)x,t)$ the natural action induced by $(\rho_t)_t$. We fix a base point $o\in X$ of the symmetric space and denote by $U_R = \cup_{\gamma\in \Gamma} B(\rho_0(\gamma)o, R)$ the $R$-neighbourhood of the $\rho_0(\Gamma)$-orbit of $o$.

Our deformation maps will be provided by the following proposition.
\begin{proposition}\label{prop:deformationmaps}
For every $R>0$ there exists a $\delta = \delta(R)>0$ smaller then $\epsilon$ and family of maps $(\Phi_t \colon U_R \to X)_{t\in D_\delta}$ satisfying the following properties:
\begin{enumerate}[label=\roman*., ref = \roman*]
\item\label[Property]{property:realanalytic} The induced map $U_R \times D_\delta \to X \colon (x, t) \mapsto \Phi_t(x)$ is real analytic.
\item\label[Property]{property:diffeos} For each $t\in D_\delta$ the set $\Phi_t(U_R)$ is open and $\Phi_t \colon U_R \to \Phi_t(U_R)$ is a real analytic diffeomorphism.
\item\label[Property]{property:identity} $\Phi_0 = \id \colon U_R \to U_R$.
\item\label[Property]{property:intertwining} For each $t\in D_\delta$ the map $\Phi_t$ intertwines the actions of $\rho_0$ and $\rho_t$ i.e. \makebox{$\rho_t(\gamma) \circ \Phi_t = \Phi_t \circ \rho_0(\gamma)$} for $\gamma\in \Gamma$.
\end{enumerate}
\end{proposition}
The content of this proposition is closely related to Ehresmann's fibration theorem. In fact, when the actions of the representations $\rho_t$ on $X$ are cocompact \Cref{prop:deformationmaps} can be obtained from it. Consequently the proof of \Cref{prop:deformationmaps} is along the same lines as the proof of the fibration theorem.

We denote by $\pr_X \colon X_\epsilon\to X$ and $\pi \colon X_\epsilon \to D_\epsilon$ the projections onto the first and second factor of $X_\epsilon = X \times D_\epsilon$ respectively. By $(t_1,...,t_n)$ we denote the coordinates on $D_\epsilon$ and also the coordinates on the $D_\epsilon$ factor in $X_\epsilon$. So in this notation we have $d\pi(\ddt[i](x,t)) = \ddt[i](t)$.

We first prove the following lemma.
\begin{lemma}\label{lem:constructionvectorfields}
Let $R>0$. On an $\alpha(\Gamma)$-invariant neighbourhood of $U_R\times\{0\}$ in $X_\epsilon$ there exist $\alpha(\Gamma)$-invariant real analytic vector fields $\xi_i$ (for $i=1,..., n$) that satisfy $d\pi(\xi_i(x,t)) = \ddt[i](t)$.
\end{lemma}
\begin{proof}
It is possible to give a more or less explicit construction for such vector fields. However proving they are indeed real analytic is rather cumbersome. Instead we opt to explicitly construct smooth vector fields which we then approximate by real analytic ones.

We let $\varphi \colon [0,\infty) \to [0,1]$ be a smooth function satisfying $\varphi\vert_{[0,R]} \equiv 1$ and $\varphi\vert_{[R+1, \infty)} \equiv 0$. For $i=1,...,n$ we define smooth vector fields $\eta_i$ on $X_\epsilon$ by
\[
\eta_i(x,t) = \varphi(d(o,x)) \cdot \ddt[i].
\]
Now let
\[
\xi_i' = \sum_{\gamma\in \Gamma} (\alpha(\gamma))_* \eta_i.
\]
The sum on the right hand side is locally finite by the uniform properness assumption on $(\rho_t)_t$. Hence each $\xi_i'$ is $\alpha(\Gamma)$-invariant smooth vector field on $X_\epsilon$. We observe that $d\pi(\eta_i(x,t)) = s(x,t) \ddt[i]$ with 
\[
s(x,t) = \sum_{\gamma\in \Gamma} \varphi(d(o, \rho_t(\gamma)^{-1}x)).
\]
On $B(o,R)\times \{0\}$ we have that $s(x,t) \geq \varphi(d(o,x)) = 1$ and by $\alpha(\Gamma)$-invariance we have that $s \geq 1$ on $U_R\times\{0\}$. 

We now approximate the smooth vector fields $\xi_i'$ by real analytic ones. By the uniformly free and proper assumption on $(\rho_t)_{t}$ we have that $X_\epsilon/\alpha(\Gamma)$ is a real analytic manifold. The vector fields $\xi_i'$ descend to smooth vector fields. On compact subsets these vector fields can be approximated arbitrarily closely in $C^0$ norm by real analytic vector fields (see \cite{Whitney} and \cite{Royden}). The set $U_R\times\{0\}$ maps to a precompact subset of $X_\epsilon/\alpha(\Gamma)$. Hence by pulling back approximating vector fields we see that on a neighbourhood of $U_R\times \{0\}$ we can approximate $\xi_i'$ arbitrarily closely by $\alpha(\Gamma)$-invariant real analytic vector fields. Let $\xi_i''$ be such approximating vector fields. For some real analytic functions $s_i'$ we have $d\pi(\xi_i''(x,t)) = s_i'(x,t)\ddt[i]$. By choosing the approximating vector fields $\xi_i''$ close enough to $\xi_i'$ we can arrange that each $s_i'$ is close to $s$ and hence satisfies $s_i'>0$ on a neighbourhood of $U_R\times\{0\}$. For $i=1,...,n$ we can now define $\xi_i = \xi_i''/s'$.
\end{proof}
\begin{proof}[Proof of \Cref{prop:deformationmaps}]
Let $\xi_i$ for $i=1,..., n$ be the vector fields constructed in \Cref{lem:constructionvectorfields}. We denote by $\psi_i^s$ their flows which are defined on a neighbourhood of $U_R\times \{0\}$. We consider the maximal flow domain for a combination of these flows starting at points in $X\times \{0\}$ i.e. the set
\[
\Omega = \{(x,s)\in X\times \R^n \mid \psi_1^{s_1}\circ \cdots \circ \psi_n^{s_n}((x,0)) \text{ is defined}\}.
\]
This is an open set containing $U_R \times \{0\}$. On $\Omega$ we set
\[
\Psi(x,(s_1,...,s_n)) = \psi_1^{s_1}\circ \cdots \circ \psi_n^{s_n}((x,0)).
\]
Because $d\pi(\xi_i) = \ddt[i]$ (when defined) we see that 
\[
t \mapsto \pi\circ \Psi(x, (s_1,..., s_{i-1}, t, s_{i+1},..., s_n))
\]
is an integral line for the vector field $\ddt[i]$. Since these integral lines are unique and $\pi\circ \Psi(x,0) = 0$ we find $\pi\circ \Psi(x,s) = s$ when $(x,s)\in \Omega$ i.e. $\pi \circ \Psi = \pi$. 

Because the vector fields $\xi_i$ are $\alpha(\Gamma)$-invariant we observe for $\gamma\in \Gamma$ that
\begin{align*}
\Psi(\rho_0(\gamma)x, s) &= \psi_1^{s_1}\circ \cdots \circ \psi_n^{s_n}(\alpha(\gamma)(x,0)) \\&= \alpha(\gamma) [\psi_1^{s_1}\circ \cdots \circ \psi_n^{s_n}(x,0)] = \alpha(\gamma) \Psi(x,s)
\end{align*}
whenever both sides are defined. We define $\beta \colon \Gamma \times \Omega \to \Omega$ as an action of $\Gamma$ on $\Omega$ by $\beta(\gamma)(x,s) = (\rho_0(\gamma)x, s)$ which is the action of $\rho_0(\Gamma)$ on $X$ times the trivial action. By the above the we see that the set $\Omega$ is $\beta(\Gamma)$-invariant and $\Psi$ intertwines the $\beta$ and $\alpha$ actions i.e. $\alpha(\gamma) \circ \Psi = \Psi \circ \beta(\gamma)$. 

On $X\times\{0\} \cap \Omega$ the map $\Psi$ is simply the inclusion into $X_\epsilon$. Combined with the fact that $\pi\circ\Psi = \pi$ we see for each $(x,0) \in X\times \{0\} \cap \Omega$ the tangent map $d\Psi\vert_{(x,0)} \colon T_x X \times T_0\R^n \to T_x X \times T_0 D_\epsilon$ is the identity map. Hence we can shrink $\Omega$ to a smaller open neighbourhood of $U_R \times \{0\}$ such that $\Psi$ is a local diffeomorphism on $\Omega$. By shrinking $\Omega$ further we can also assume $\Psi$ to be injective. For if not then there exist two distinct sequences $(x_n,s_n), (x_n',s_n')\in \Omega$ satisfying $\Psi(x_n,s_n) = \Psi(x_n', s_n')$ with $s_n, s_n'\to 0$ and $x_n, x_n'$ converging to points $x$ and $x'$ in $U_R$. By $\pi\circ\Psi=\pi$ we see $s_n=s_n'$. By continuity $\Psi(x, 0) = \Psi(x',0)$ and because when restricted to $X\times\{0\}\cap \Omega$ the map $\Psi$ is an injection we must have $x=x'$. However this contradicts the fact that $\Psi$ is a local diffeomorphism. So we can arrange that $\Psi$ is a diffeomorphism onto its image. Since $\Psi$ intertwines $\beta$ and $\alpha$ this can be done such that $\Omega$ is still $\beta(\Gamma)$-invariant.

Since $\Omega$ is a neighbourhood of $U_R\times \{0\}$ we can, by compactness, find a $\delta>0$ such that $\overline{B(o,R)}\times D_\delta \subset \Omega$. By $\beta(\Gamma)$-invariance we then have $U_R \times D_\delta \subset \Omega$. We now define the family of deformation maps $\Phi_t \colon U_R \to X$ as $\Phi_t(x) = \pr_{X}\circ \Psi(x,t)$ for $t\in D_\delta$. We check that indeed $(\Phi_t)_{t\in D_\delta}$ satisfies Properties (\ref{property:realanalytic})-(\ref{property:intertwining}). Property (\ref{property:realanalytic}) follows since flows of real analytic vector fields are real analytic. Property (\ref{property:diffeos}) follows since $\Psi \colon \Omega \to \Psi(\Omega)$ is a diffeomorphism and satisfies $\pi \circ \Psi = \pi$ hence induces diffeomorphisms between the fibres $\pi^{-1}(t)\cap \Omega$ and $\pi^{-1}(t) \cap \Psi(\Omega)$. Property (\ref{property:identity}) follows from the fact that $\Psi$ restricted to $X\times\{0\}\cap \Omega$ is the inclusion map and Property (\ref{property:intertwining}) follows from the fact that $\Psi$ intertwines the actions of $\beta$ and $\alpha$.
\end{proof}

Using the deformation maps the problem of dependence on representations can be reduced to the problem of dependence on metrics on a fixed target manifold. In this case the results of \cite{EellsLemaire} can be used. In their paper Eells and Lemaire only prove smooth dependence so for completeness we prove a version of their result in the real analytic category.
\begin{proposition}\label{prop:implicitfunctiontheorem}
Let $M, N$ be real analytic manifolds with $M$ compact. Let $(g_t)_{t\in D_\epsilon}, (h_t)_{t\in D_\epsilon}$ be real analytic families of metrics on $M$ respectively $N$. If $f_0 \colon (M,g_0) \to (N,h_0)$ is a harmonic map such that $\nabla^2 E(f_0)$ is non-degenerate then for every $k\in \N, 0 < \alpha < 1$ there exists a $\delta>0$ and a unique real analytic map $F \colon D_\delta \to C^{k,\alpha}(M, N)$ such that $F(0) = f_0$ and each $F(t)$ is a harmonic map $(M,g_t) \to (N,h_t)$. 
\end{proposition}
\begin{proof}
For each $t\in D_\epsilon$ a $C^2$ map $\phi \colon (M,g_t) \to (N,h_t)$ is harmonic if ${\tau_t(\phi) = \tr_{g_t} \nabla d\phi=0}$ where $\nabla$ is the connection on $T^*M \otimes \phi^* TN$ induced by $g_t$ and $h_t$. In local coordinates $(x^i)_i$ on $M$ and $(u^\alpha)_\alpha$ on $N$, $\tau_t(\phi)$ is given by
\[
\tau_t(\phi)^\gamma = (g_t)_{ij}
\left\lbrace
\parder{^2 \phi^\gamma}{x^i \partial x^j}
- (\Gamma_{g_t})_{ij}^k \parder{\phi^\gamma}{x^k} + (\Gamma_{h_t})_{\alpha\beta}^\gamma (\phi)
\parder{\phi^\alpha}{x_i}\parder{\phi^\beta}{x_j}
\right\rbrace
\]
here $\Gamma_{g}$ denote the Christoffel symbols of a metric $g$. We combine the tension fields for different $t\in D_\epsilon$ in a map
\[
\tau \colon C^{k+2+\alpha}(M,N) \times D_\epsilon \to C^{k+\alpha}(M,TN).
\]
We claim this map is real analytic. To see this we write $\tau$ as a composition of two real analytic map. First we consider the second prolongation map
\[
J \colon C^{k+2+\alpha}(M,N) \to C^{k+\alpha}(M, J^2(M,N))
\]
mapping a map $\phi \colon M \to N$ to its 2-jet $j^2\phi$. A diffeomorphism between a neighbourhood of the zero section in $\phi^*TN$ and a neighbourhood of the image of $\text{graph}(\phi)$ in $M\times N$ induces charts of the two mapping spaces modelled on $\Gamma^{k+2+\alpha}(\phi^*TN)$ and $\Gamma^{k+\alpha}(J^2(M,\phi^*TN))$ respectively. In these charts the second prolongation map is a bounded linear map so in particular it is real analytic. Secondly, we consider the map $T\colon J^2(M,N) \times D_\epsilon \to TN$ given in local coordinates (induced by $(x^i)_i$ on $M$ and $(u^\alpha)_\alpha$ on $N$) by
\[
(x^i, u^\alpha, u^{\alpha}_{i}, u^{\alpha}_{ij}, t) \mapsto (g_t)_{ij}
\left\lbrace
u^\gamma_{ij}
- (\Gamma_{g_t})_{ij}^k u^\gamma_{k} + (\Gamma_{h_t})_{\alpha\beta}^\gamma (u)
u^\alpha_i u^\beta_j
\right\rbrace
\]
By assumption the coefficients $(g_t)_{ij}, (\Gamma_{g_t})_{ij}^k$ and $(\Gamma_{h_t})_{\alpha\beta}^\gamma$ are real analytic functions so the map $T$ is real analytic. It now follows from the $\Omega$-lemma (see \Cref{lem:omegalemma} below) that $T$ induces a real analytic map
\begin{align*}
T_*\colon &C^{k+\alpha}(M, J^2(M,N))\times D_\epsilon \to C^{k+\alpha}(M,TN)\\
&(\psi,t) \mapsto T(\psi(\cdot),t)
\end{align*}
The map $\tau$ is a composition of $T_*$ and $J$ and is therefore real analytic.

As discussed in \cite{EellsLemaire}[p. 35] the partial derivative of $\tau$ with respect to the first factor $(D_1 \tau)_{(f_0,0)} \colon T C^{k+2+\alpha}(M,N) \to T_0 T C^{k+\alpha}(T,N)$ is an isomorphism of Banach spaces precisely when $\nabla^2 E(f_0)$ is non-degenerate. Hence we can apply a real analytic version of the implicit function theorem for Banach spaces (e.g. \cite{Whittlesey}) to obtain, for $\delta>0$ small enough, a unique real analytic map $F \colon D_\delta \to C^{k+2+\alpha}(M,N)$ such that $F(0) = f_0$ and $\tau_t(F(t)) = 0$ for all $t\in D_\delta$.
\end{proof}
\begin{lemma}[The $\Omega$-Lemma]\label{lem:omegalemma}
Let $M, N$ and $P$ be real analytic manifolds with $M$ compact. Suppose $F \colon N \to P$ is real analytic. Then $F$ induces a real analytic map
\[
\Omega_F \colon C^{k,\alpha}(M,N) \to C^{k,\alpha}(M, P) \colon \phi \mapsto F \circ \phi
\]
for all $k\in \N, 0 < \alpha < 1$.
\end{lemma}
Compare with \cite[Theorem 11.3]{Abraham}. Unfortunately a proof of the real analytic case as stated here does not seem to be available in the literature. We give a sketch of the proof.

\begin{proof}[Sketch of proof of \Cref{lem:omegalemma}]
By following the same steps as in \cite{Abraham} the statement can be reduced to a local version (c.f \cite[Theorem 3.7]{Abraham}) i.e. it is enough to prove that if $K \subset \R^n$ is compact, $V\subset \R^p$ open and $F \colon K \times V \to \R^q$ real analytic then $\Omega_F \colon C^{k,\alpha}(K, V) \to C^{k,\alpha}(K, \R^q)$ given by $[\Omega_F(\phi)](x) = F(x, \phi(x))$ is a real analytic map between Banach spaces. To this end we observe that since $F$ is real analytic it can be extended to a complex analytic map $\widetilde{F} \colon U \to \C^q$ on an open set $U \subset \C^n \times \C^p$ containing $K \times V$. Let $\widetilde{V} \subset \C^p$ be an open such that $K \times V \subset K\times \widetilde{V} \subset U$. Then $\widetilde{F}$ induces a map $\Omega_{\widetilde{F}} \colon C^{k,\alpha}(K, \widetilde{V}) \to C^{k,\alpha}(K, \C^q)$ between complex Banach spaces which extends $\Omega_F$. Applying the smooth version of the $\Omega$-Lemma yields that $\Omega_{\widetilde{F}}$ is a $C^1$ map with derivative given by $D\Omega_{\widetilde{F}} = \Omega_{D_2 \widetilde{F}}$. Since $\widetilde{F}$ is holomorphic we see that this derivative is complex linear. It now follows from \cite[Theorem A5.3]{Hubbard} that $ \Omega_{\widetilde{F}}$ is a complex analytic map. We conclude that $\Omega_F$, which is a restriction of $\Omega_{\widetilde{F}}$ to $C^{k,\alpha}(K, V)$, is real analytic.\end{proof}

We can now prove the statement of our main theorem.
\begin{proof}[Proof of \Cref{thm:mainresult}]
We set $N = X/\rho_0(\Gamma)$ (recall that by assumption the action of $\rho_0$ on $X$ is free and proper so $N$ is a manifold). Since $\rho_0$ is reductive and $Z_G(\im \rho_0)$ contains no semi-simple elements there exists a unique $\rho_0$-equivariant harmonic map $f\colon \widetilde{M} \to X$. This map descends to a harmonic map $f \colon M \to N$. We denote by $o'$ the point in $N$ covered by the base point $o$ in $X$. The set $U_R$ descends to the set $V = B(o', R)$ in $N$. We choose $R>0$ large enough such that the image of $f$ is contained in $B(o', R)$.

Let $(\Phi_t)_{t\in D_\delta}$ be the family of deformation maps as in \Cref{prop:deformationmaps}. We denote by $m$ the Riemannian metric on the symmetric space $X$. Define a family of metrics $(h_t)_{t\in D_\delta}$ on $U_R$ by $h_t = \Phi_t^* m$. By Property (\ref{prop:deformationmaps}.\ref{property:realanalytic}) this is a real analytic family of metrics. We observe that for $\gamma\in \Gamma$
\begin{align*}
\rho_0(\gamma)^* h_t = \rho_0^*\Phi_t^* m = \Phi_t^* \rho_t(\gamma)^* m = \Phi_t^* m = h_t
\end{align*}
using Property (\ref{prop:deformationmaps}.\ref{property:intertwining}) and the fact that each $\rho_t$ acts by isometries on $X$. We conclude that each $h_t$ is $\rho_0(\Gamma)$-invariant hence the family of metrics descends to a family of metrics, also denoted $(h_t)_{t\in D_\delta}$, on $V$. By \Cref{lem:nondegeneracy} the Hessian $\nabla^2E(f)$ is non-degenerate so \Cref{prop:implicitfunctiontheorem} yields, after shrinking $\delta$, a unique real analytic map $G \colon D_\delta \to C^{k,\alpha}(M, V)$ such that $G(t)$ is a harmonic map from $(M,g_t)$ to $(V, h_t)$ for each $t\in D_\delta$. By choosing for each $t$ a $\rho_0$-equivariant lift we can view $G$ as a continuous map $G \colon D_\delta \to C^{k,\alpha}(\widetilde{M}, U_R)$. We define $F$ by composing with the deformation maps, $F(t) = \Phi_t \circ G(t)$. By Property (\ref{prop:deformationmaps}.\ref{property:intertwining}) every map $F(t)$ is $\rho_t$-equivariant. By construction, each $\Phi_t$ is an open isometric embedding of $(V, h_t)$ into $(X,m)$ hence each $F(t)$ is also harmonic. Finally by Property (\ref{prop:deformationmaps}.\ref{property:realanalytic}) we see that the map $F \colon D_\delta \to C^{k,\alpha}(\widetilde{M}, X)$ is continuous and real analytic as a map $F(\cdot)\vert_\Omega \colon D_\delta \to C^{k,\alpha}(\Omega,X)$.
\end{proof}

\bibliographystyle{alpha}
\bibliography{bibliography}

\end{document}